\documentclass[11pt]{amsart}
\usepackage{amssymb,amsmath,amsthm}
\usepackage{color,titletoc}
\makeindex

\usepackage[pagebackref,colorlinks,linkcolor=red,citecolor=blue,urlcolor=blue,hypertexnames=true]{hyperref}

\usepackage{enumerate}

\DeclareMathOperator{\Aut}{Aut}

\swapnumbers

\newtheorem{theorem}[subsection]{Theorem}
\newtheorem{corollary}[subsection]{Corollary}
\newtheorem{lemma}[subsection]{Lemma}
\newtheorem{proposition}[subsection]{Proposition}

\newtheorem{remark}[subsection]{Remark}

\theoremstyle{definition}
\newtheorem{defn}[subsection]{Definition}

\theoremstyle{remark}

 \def\bCx{\mathbb C\langle x_1,\dots,x_g\rangle}
 \def\bCxshort{\mathbb C\langle x\rangle}

\def\cFg{\langle x\rangle}
 \def\cFgtwo{\mathcal F_g}

\def\e{ \exp(i\theta) }
\def\cR{\mathcal R}
 
 \def\cD{\mathcal D}

 \def\cN{\mathcal N}
 \def\matng{M_n(\mathbb C)^g}
 \def\matnh{M_n(\mathbb C)^{\h}}
 \def\cL{\mathcal L}

 \def\h{\tilde{g}}
 \def\math{M(\mathbb C)^{\h}}
 \def\matg{M(\mathbb C)^g}
 \def\cU{\mathcal U}
 \def\cH{\mathcal H}
 \def\cV{\mathcal V}
 \def\cP{\mathcal P}
 
 \def\vy{\mathbf y}
 \def\cW{\mathcal W}
 \def\cY{\mathcal Y}

\def\RR{\mathbb R}
\def\R{\mathbb R}

\def\NN{\mathbb N}
\def\N{\mathbb N}
\def\C{\mathbb C}

\newcommand{\eps}{\epsilon}
\def\ben{\begin{enumerate} }
\def\een{\end{enumerate} }
\def\beq{\begin{equation} }
\def\eeq{\end{equation} }

\def\tg{\tilde g}

\def\fm{f^{(m)}}

\def\nca{free map }
\def\ncac{free map, }
\def\ncas{free  maps }
\def\ncasp{free  maps. }

\def\ncap{free  map. }
\def\ncacomma{free  map, }

\def\cncap{free analytic   map. }
\def\cncasp{free analytic   maps. }
\def\cncasc{free analytic   maps, }
\def\cncacomman{ free  analytic map, }
\def\cnca{free analytic map }

\def\cncas{free analytic maps }

\def\fba{free biholomorphism }
\def\fbap{free biholomorphism. }
\def\fbasp{free biholomorphisms. }

\def\fn{f[n]}

\def\fm{f[m]}
\def\fone{f[1]}

\def\fhm{f^{(m)}}
\def\fhmn{\fhm[n]}

\def\Rx{\mathbb R\langle x \rangle} 
\def\SRng{\mathbb S_n^g}

\renewcommand{\setminus}{\smallsetminus}

\renewcommand{\theequation}{(\arabic{equation})}

\makeatletter
\renewcommand{\c@equation}{\c@subsection}
\makeatother

\numberwithin{equation}{section}

\makeatletter
  \def\@eqnnum{{\normalfont \normalcolor \theequation}}
\makeatother

\makeatletter
  \def\tagform@#1{\maketag@@@{#1\@@italiccorr}}
\makeatother

\begin{document}

\setcounter{page}{1}

\title[Free analysis and LMI domains]{Free analysis, convexity and LMI domains}

\author[Helton]{J. William Helton${}^1$}
\address{J. William Helton\\ Department of Mathematics\\
  University of California \\
  San Diego}
\email{helton@math.ucsd.edu}
\thanks{${}^1$Research supported by NSF grants
DMS-0700758, DMS-0757212, and the Ford Motor Co.}

\author[Klep]{Igor Klep${}^2$}
\address{Igor Klep\\ Department of Mathematics\\
The University of Auckland, New Zealand
}
\email{igor.klep@auckland.ac.nz}
\thanks{${}^2$Research supported by the Slovenian Research Agency grants
J1-3608 and P1-0222. The article was written while the author was on leave from 
the University of Maribor.}

\author[McCullough]{Scott McCullough${}^3$}
\address{Scott McCullough\\ Department of Mathematics\\
  University of Florida, Gainesville 
   }
   \email{sam@math.ufl.edu}
\thanks{${}^3$Research supported by the NSF grant DMS-0758306.}



\date{\today}

\setcounter{tocdepth}{2}
\contentsmargin{2.55em} 
\dottedcontents{section}[3.8em]{}{2.3em}{.4pc} 
\dottedcontents{subsection}[6.1em]{}{3.2em}{.4pc}

\makeatletter
\newcommand{\mycontentsbox}{%
{\centerline{NOT FOR PUBLICATION}
\small\tableofcontents}}
\def\enddoc@text{\ifx\@empty\@translators \else\@settranslators\fi
\ifx\@empty\addresses \else\@setaddresses\fi
\newpage\mycontentsbox}
\makeatother

\begin{abstract} 
This paper concerns  {\it free analytic  maps}
on {\it noncommutative domains}. These 
 maps
are free analogs of classical holomorphic functions in several
complex variables, and
 are defined in terms of noncommuting variables
amongst which there are no relations - they are 
{\em free} variables. 
Free analytic  maps
include vector-valued
polynomials in free (noncommuting) variables and
form a canonical class of mappings from
one noncommutative domain $\cD$ in say $g$ variables
to another noncommutative domain $\tilde \cD$ in $\tg$ variables.

Motivated by determining the possibilities for mapping a nonconvex noncommutative domain to a convex noncommutative domain, this
article focuses on rigidity results for free analytic maps.  Those obtained 
to date, parallel and are often stronger than those in several complex
variables.  For instance, a proper free analytic map between noncommutative domains
is one-one and, if $\tg=g$, free biholomorphic.  Making its debut
here is a free version of a theorem of Braun-Kaup-Upmeier:
between two freely biholomorphic bounded circular
noncommutative domains there exists a {\em linear} biholomorphism.
An immediate consequence is the following {\em nonconvexification} result:
 if two bounded circular noncommutative domains are 
  freely biholomorphic, then they are either both convex or both not convex. 
Because of their roles in systems engineering,
\emph{linear matrix inequalities} (LMIs) and 
noncommutative domains 
defined by an LMI ({\em LMI domains})
 are of particular interest.  As a refinement of above the nonconvexification result,
 if a bounded circular noncommutative
 domain $\cD$ is freely biholomorphic to a bounded circular
 LMI domain, then $\cD$ is itself an LMI domain. 
\end{abstract}

\maketitle

\section{Introduction}
 \label{sec:intro}
  The notion of an analytic, free  or noncommutative,
  map arises naturally in free
  probability,
  the study of noncommutative (free) rational functions
  \cite{BGM,Vo1,Vo2,AIM,MS11, KVV},
 and
systems theory \cite{HBJP}.
  In this paper
  rigidity results for such functions  paralleling those
  for their classical commutative counterparts are presented.
 Often in the noncommutative (nc) setting such theorems have cleaner 
 statements than their commutative counterparts.
 Among these we shall present the following:

\ben[\rm (1)]
\item
  a {\em continuous} free map is {\em analytic} (\S\ref{sec:contAnal}) 
  and hence admits
  a {\em power series} expansion (\S\ref{sec:powerseries});
\item
 if $f$ is a {\em proper} analytic free  map from
  a noncommutative domain in $g$ variables
  to another in $\tg$ variables, then $f$ is {\em injective}
  and $\tg \ge g$. 
  If in addition 
  $\tg=g$, then  $f$  is onto and
  has an inverse which is itself
  a (proper) analytic free map (\S\ref{sec:properInj}).
 This injectivity conclusion contrasts markedly with the classical
 case where a
 (commutative) {\em proper} 
  analytic function $f$ from one domain in $\mathbb C^g$
  to another in $\mathbb C^g,$  need not be injective,
  although it must be onto.
\item
  A free Braun-Kaup-Upmeier theorem (\S\ref{sec:ku}). 
  A free analytic map $f$ is called a free biholomorphism if
  $f$ has an inverse $f^{-1}$ which is also a
  free analytic map.  As an extension of a
  theorem from \cite{BKU}, two bounded,  circular, noncommutative domains
  are freely {\em biholomorphic} if and only if they are freely
  {\em linearly} biholomorphic. 
\item
  Of special interest are free analytic mappings from or to or both from and to
  noncommutative domains defined by linear matrix inequalities, or LMI 
  domains.  Several additional  recent results in this direction, as well as 
  a concomitant free {\em convex Positivstellensatz} (\S\ref{sec:Posss}), are also included. 
\een
  Thus this article is largely a survey.  The results of items
 (1), (2), and (4) appear elsewhere. However, 
the main result of (3) is new. 
 Its proof relies on the existence of power series expansions for
 analytic free maps,
a topic we discuss as part of (1) in \S\ref{sec:powerseries} below.  
 Our treatment is modestly different from  that found in
 \cite{Vo2,KVV}.

 For the classical theory of 
  commutative proper analytic maps
  see D'Angelo \cite{DAn} or Forstneri\v c \cite{For}.
We assume the reader is familiar with basics of several
complex variables as given e.g.~in Krantz \cite{Krantz}.

\subsection{Motivation}\label{subsec:motiv}
One of the main advances in systems engineering in the 1990's
 was the conversion of a set of problems to 
{\em linear matrix inequalities} ({\em LMIs}),
 since LMIs, up to modest size, can be solved numerically
 by semidefinite programs \cite{SIG97}.
 A large  class of linear
 systems problems are described in terms of
 a signal-flow diagram $\Sigma$ plus $L^2$
 constraints (such as energy dissipation).
 Routine methods convert such problems into
 noncommutative polynomial inequalities
 of the form  $p(X)\succeq 0$ or $p(X)\succ 0$.

 Instantiating
 specific  systems of linear differential
 equations  for the ``boxes'' in the system flow diagram
 amounts to substituting  their coefficient matrices
 for variables in the polynomial $p$.
 Any property  asserted to  be true must hold
 when matrices of any size are substituted into $p$.
 Such problems are referred to as {\em dimension-free}.
 We emphasize, the polynomial $p$ itself is determined
 by the signal-flow diagram $\Sigma$.

Engineers vigorously {\em seek convexity}, since optima are global
and convexity lends itself to numerics.
Indeed,  there are over  a thousand papers trying to convert
linear systems problems to convex ones and the only known technique is
the rather blunt trial and error instrument of trying to guess an LMI.
Since  having an LMI is seemingly  more restrictive than convexity,
there has been the hope, indeed expectation, that some practical
class of convex situations   has been missed.

Hence a main goal of this line of research has been to determine
which {\em changes of variables} can produce convexity
from nonconvex situations.
As we shall see below, a \cnca between noncommutative domains
cannot produce convexity from a nonconvex set, at
 least under a circularity hypothesis. 
Thus we think the implications of our results here are negative
for linear systems engineering; for dimension-free problems the 
evidence here is that 
there is no convexity beyond the obvious.

\subsection{Reader's guide}
 The definitions as used in this paper are given in the
 following section \S\ref{sec:maps}, which
contains the background on
  {\it noncommutative domains} and on {\it free maps}
  at the level of generality needed for this paper.
  As we shall see, free maps that are continuous are also analytic
(\S\ref{sec:contAnal}). 
We explain, in \S\ref{sec:powerseries}, how to associate a power series expansion
to an analytic free map using the noncommutative Fock space.
One typically thinks of free maps
  as being analytic, but in a weak sense.
In \S\ref{sec:main-result} we consider {\em proper} free maps
and give several rigidity theorems.
For instance, proper analytic free maps are injective (\S\ref{sec:properInj})
and, under mild additional assumptions, tend to be linear (see 
\S\ref{sec:analogs} and \S\ref{sec:ku}  for precise statements). 
Results paralleling classical results on analytic maps in several
complex variables, such as the Carath\'eodory-Cartan-Kaup-Wu (CCKW) Theorem, are given in
\S\ref{sec:analogs}.
 A new result - a free version of the Braun-Kaup-Upmeier (BKU) theorem - appears in 
 \S\ref{sec:ku}.  
A brief overview of further topics, including links to
references, is given in \S\ref{sec:misc}.
Most of the material presented in this paper has been motivated
by problems in systems engineering, and this was discussed briefly
above in \S\ref{subsec:motiv}.

\section{Free maps}
 \label{sec:maps}

  This section contains the background on
  noncommutative sets and on  {\it free maps}
  at the level of generality needed for this paper.
  Since power series are used in \S\ref{sec:ku},
  included at the end of this section 
  is a sketch of an argument showing that
  continuous free maps have formal power series expansions.
  The discussion borrows heavily from
the  recent basic work of
Voiculescu \cite{Vo1, Vo2} and of
Kalyuzhnyi-Verbovetski\u\i{} and Vinnikov
\cite{KVV}, see also the references therein.
These papers contain a more power series based approach
to free maps
and for more on this one can
see Popescu \cite{Po1,Po2}, or also \cite{HKMS,HKM1,HKM3}.

\subsection{Noncommutative sets and domains}
 \label{subsec:sets}
  Fix a positive integer $g$.
  Given a positive integer $n$, let $\matng$ denote
  $g$-tuples of  $n\times n$ matrices.
  Of course, $\matng$ is naturally identified with
 $M_n(\mathbb C) \otimes\mathbb C^g.$

\renewcommand{\subset}{\subseteq}
   A sequence $\cU=(\cU (n))_{n\in\N},$ where $\cU (n) \subset \matng$,
 is a {\bf noncommutative set} \index{noncommutative set}
 if it is \index{closed with respect to unitary similarity}
 {\bf closed with respect to simultaneous unitary similarity}; i.e.,
 if $X\in \cU (n)$ and $U$ is an $n\times n$ unitary matrix, then
\beq\label{eq:setU}
   U^\ast X U =(U^\ast X_1U,\dots, U^\ast X_gU)\in \cU (n);
\eeq
 and if it is \index{closed with respect to direct sums}
 {\bf closed with respect to direct sums}; i.e.,
  if $X\in \cU (n)$ and $Y\in \cU(m)$ implies
 \beq\label{eq:setS}
   X\oplus Y = \begin{bmatrix} X & 0\\ 0 & Y \end{bmatrix}
     \in \cU(n+m).
 \eeq
 
  Noncommutative sets differ from
  the fully matricial $\C^g$-sets of Voiculescu
 \cite[Section 6]{Vo1} in that the latter are closed
 with respect to simultaneous similarity, not just
    simultaneous {\em unitary} similarity.
  Remark \ref{rem:sim-vs-unitary} below  briefly discusses  the significance
  of this distinction for the results on proper analytic
  free maps in this paper.

   The noncommutative set $\cU$ is a
  {\bf noncommutative domain} if each
  $\cU (n)$ is nonempty, open and connected.
  \index{noncommutative domain}
 Of course the sequence $\matg=(\matng)$
 is  itself a noncommutative domain. Given $\varepsilon>0$,
 the set $\cN_\varepsilon = (\cN_\varepsilon(n))$ given by
\beq\label{eq:nbhd}
  \cN_\varepsilon(n)=\big\{X\in\matng : \sum X_j X_j^\ast \prec \varepsilon^2  \big\}
\eeq
 is a noncommutative domain which we
 call the
  {\bf noncommutative $\varepsilon$-neighborhood of $0$ in $\mathbb C^g$}.
  \index{noncommutative neighborhood of $0$}
  The noncommutative set $\cU$ is {\bf bounded}
  \index{bounded, noncommutative set} if there
  is a $C\in\R$ such that
\beq\label{eq:bd}
   C^2 -\sum X_j X_j^\ast \succ 0
 \eeq
   for every $n$ and $X\in\cU(n)$. Equivalently, for
   some $\lambda\in\RR$, we have $\cU\subset \cN_\lambda$.
  Note that this condition is stronger than asking
 that each $\cU(n)$ is bounded.

  Let $\bCxshort=\bCx$ denote the $\C$-algebra freely generated by $g$
noncommuting letters $x=(x_1,\ldots,x_g)$. Its elements are
linear combinations of words in $x$ and are called
(analytic)
{\bf polynomials}.
  Given an $r\times r$ matrix-valued polynomial
  $p\in M_r(\mathbb C) \otimes \bCx$ with
  $p(0)=0$, let $\cD(n)$ denote the connected
  component of
\beq\label{eq:saset}
  \{X\in \matng : I+p(X)+p(X)^*  \succ 0\}
\eeq
  containing the origin.
  The sequence $\cD=(\cD(n))$ is a noncommutative domain
  which is semi-algebraic in nature.
  Note that $\cD$ contains an $\varepsilon>0$ neighborhood
  of $0,$ and that the choice
 \[
   p=  \frac{1}{\varepsilon}
    \begin{bmatrix}\;\; 0_{g\times g}  & \begin{matrix} x_1 \\ \vdots \\x_g\end{matrix} \\
            \;\;  0_{1\times g}  & 0_{1\times 1} \end{bmatrix}
 \]
  gives $\cD = \cN_\varepsilon$.
  Further examples of natural noncommutative domains
  can be generated by considering noncommutative polynomials
  in both the variables $x=(x_1,\dots,x_g)$
  and their formal adjoints, $x^*=(x_1^*,\dots,x_g^*)$.
  For us the motivating case of domains is determined
  by linear matrix inequalities (LMIs).

\subsection{LMI domains}
  A special case of the noncommutative domains
  are those described by a linear matrix inequality. 
  Given a positive integer $d$
  and  $A_1,\dots,A_g \in M_d(\C),$ the
  linear matrix-valued polynomial
\beq\label{eq:linPenc}
  L(x)=\sum A_j x_j\in M_d(\C)\otimes \bCx
\eeq
  is a {\bf (homogeneous) linear pencil}.\index{homogeneous linear pencil}
  Its adjoint is, by definition,
$
  L(x)^*=\sum A_j^* x_j^*.
$
  Let
\[
  \cL(x) = I_d + L(x) +L(x)^*.
\]
  If $X\in \matng$, then $\cL(X)$ is defined by the canonical substitution,
\[
  \cL(X) = I_d\otimes I_n
          +\sum A_j\otimes X_j +\sum A_j^* \otimes X_j^*,
\]
 and yields a symmetric $dn\times dn$ matrix.
  The inequality $\cL(X)\succ 0$ for tuples $X\in\matg$
  is a {\bf linear matrix inequality (LMI)}. \index{LMI}
  \index{linear matrix inequality} The sequence of solution sets
 $\cD_{\cL}$ defined by
\beq\label{eq:LMI}
 \cD_{\cL}(n) = \{X\in\matng : \cL(X)\succ 0\}
\eeq
 is a noncommutative domain which contains a neighborhood
 of $0$. It is called a {\bf noncommutative (nc) LMI domain}.
 It is also a particular instance of 
a noncommutative semialgebraic set.

\subsection{Free mappings}
 \label{subsec:nc-maps}
 Let $\cU$ denote a noncommutative subset of $\matg$
 and let $\h$ be a positive integer.
 A
 {\bf \nca}\index{\nca}
  $f$ from $\cU$ into $\math$ is a sequence
 of functions $\fn:\cU(n) \to\matnh$ which
{\bf respects direct sums}:
  for each $n,m$ and $X\in\cU(n)$ and  $Y\in \cU(m),$
 \beq\label{eq:mapsS}
   f(X\oplus Y)=f(X) \oplus f(Y);
 \eeq
and
   {\bf respects similarity}: for each $n$ and
  $X,Y\in \cU(n)$ and invertible $n\times n$
  matrix $\Gamma$
  such that 
\beq
  X\Gamma = (X_1\Gamma, \dots, X_g\Gamma)
   = (\Gamma Y_1,\dots, \Gamma Y_g) =\Gamma Y
\eeq
we have
\beq\label{eq:mapsU}
   f(X) \Gamma = \Gamma f(Y).
\eeq
  Note if $X\in\cU(n)$ it is natural
  to write simply $f(X)$ instead of
 the more cumbersome $\fn(X)$ and
 likewise $f:\cU\to \math$.

We say $f$
  {\bf respects intertwining maps}
 if $X\in\cU(n)$, $Y\in\cU(m)$, $\Gamma:\mathbb C^m\to\mathbb C^n$,
  and  $X\Gamma=\Gamma Y$
  implies $\fn(X) \Gamma =  \Gamma \fm (Y)$.
  The following proposition gives an alternate characterization
  of \ncasp 
Its easy proof is left to the reader (alternately,
see \cite[Proposition 2.2]{HKM3}).

 \begin{proposition}
  \label{prop:nc-map-alt}
   Suppose $\cU$ is a noncommutative subset of $\matg$.  A sequence
   $f=(\fn)$  of functions $\fn:\cU(n)\to \matnh$
  is a \nca if and only if it
  respects intertwining maps.
 \end{proposition}

\begin{remark}\rm
 \label{rem:sim-vs-unitary}
   Let $\cU$ be a noncommutative domain in $M(\C)^g$ and
  suppose $f:\cU\to M(\C)^{\tg}$ is a free map. If $X\in \cU$
  is similar to $Y$ with $Y=S^{-1}X S$, then we can 
  define $f(Y) = S^{-1}f(X)S$. In this way $f$ naturally
 extends to a free map on $\cH(\cU)\subset M(\C)^g$ defined by
\[
  \cH(\cU)(n)=\{Y\in M_n(\C)^g: \text{ there is an } X\in \cU(n)
     \text{ such that $Y$ is similar to $X$}\}.
\]
  Thus if $\cU$ is  a domain of holomorphy, then $\cH(\cU)=\cU$.

 On the other hand, because our results on proper analytic free maps
  to come depend strongly upon 
  the  noncommutative set $\cU$ itself,
   the distinction between noncommutative sets and
   fully matricial sets as in  \cite{Vo1} is important.
    See also  \cite{HM,HKM2,HKM3}.
\end{remark}

 We close this subsection with a simple observation:

\begin{proposition}
 \label{prop:range}
  If $\cU$ is a noncommutative subset of $\matg$ and
  $f:\cU\to \math$ is a \ncacomma then the range of $f$, equal to
  the sequence $f(\cU)=\big(f[n](\cU(n))\big)$, is itself
  a noncommutative subset of $\math$.
\end{proposition}

\subsection{A continuous \nca is analytic}\label{sec:contAnal}
  Let $\cU\subset \matg$ be a noncommutative set.
  A \nca   $f:\cU\to \math$ is {\bf continuous} if each
  $\fn:\cU(n)\to \matnh$ is continuous. \index{continuous}
  Likewise,
  if $\cU$ is a noncommutative domain, then
  $f$ is called {\bf analytic} if each $\fn$ is analytic.
 \index{analytic}
  This implies the existence of directional derivatives for all directions
 at each point in the domain, and this is the property we 
use often.
Somewhat surprising, though easy to prove, is the following:

\begin{proposition}
 \label{prop:continuous-analytic}
  Suppose $\cU$ is a noncommutative domain in $\matg$.
\ben[\rm (1)]
\item
  A continuous \nca $f:\cU\to \math$ is analytic.
\item If $X\in\cU(n)$, and $H\in\matng$ has sufficiently small norm,
then
 \beq\label{eq:propcontanal}
  f\begin{bmatrix} X & H \\ 0 & X\end{bmatrix}
   = \begin{bmatrix} f(X) & f^\prime(X)[H] \\ 0 & f(X)\end{bmatrix}.
 \eeq
\een
\end{proposition}

We shall not prove this here and refer the reader to \cite[Proposition 2.5]{HKM3} for a proof. The equation \eqref{eq:propcontanal} 
appearing in item (2) will be greatly
expanded upon in \S \ref{sec:powerseries} immediately below, where we 
explain how every \cnca admits a convergent power series expansion.

\subsection{Analytic free maps have a power series expansion}\label{sec:powerseries}

 It is shown in \cite[Section 13]{Vo2}
 that a \cnca
 $f$ has a formal power series expansion
 in the noncommuting variables, which indeed is a powerful way
 to think of \cncasp  Voiculescu also gives elegant
 formulas for the coefficients of the power series expansion of $f$ in terms
 of clever evaluations of $f$. Convergence properties for bounded
\cncas are studied in \cite[Sections 14-16]{Vo2}; see also
\cite[Section 17]{Vo2} for a bad unbounded example. 
Also,
  Kalyuzhnyi-Verbovetski\u\i{} and Vinnikov \cite{KVV} are developing
 general results based on very weak hypotheses with the conclusion that
 $f$ has a power series expansion and is thus a \cncap
An early study of noncommutative mappings is given
in \cite{Taylor}; see also \cite{Vo1}.

 Given a positive integer $\tg$, 
 a {\bf formal power series} $F$
 in the variables $x=\{x_1,\dots,x_g\}$
 with coefficients in $\mathbb C^{\tg}$
 is an expression of the form
\[
   F=\sum_{w\in\cFg} F_w w
\]
 where the $F_w\in \mathbb C^{\tg}$, and
 $\cFg$ is the free monoid on $x$, i.e., the set
 of all words in the noncommuting variables $x$.
 (More generally, the $F_w$ could be chosen
  to be operators between two Hilbert spaces.
  With the choice of $F_w\in\mathbb C^{\tg}$ and with some
  mild additional hypothesis,
  the power series $F$ determines a
  free map from some noncommutative
  $\varepsilon$-neighborhood of $0$ in
  $M(\C)^g$ into $M(\C)^{\tg}$.)

 Letting $F^{(m)}=\sum_{|w|=m} F_w w$ denote the
 {\bf homogeneous of degree $m$ part}
 of $F$,  \index{homogeneous of degree $m$ part}
\begin{equation}
 \label{eq:formal-pow}
  F=\sum_{m=0}^\infty \sum_{|w|=m} F_w w =\sum_m F^{(m)}.
\end{equation}

\begin{proposition}
 \label{prop:pow-nc-map}
 Let $\cV$ denote a noncommutative domain in
 $\matg$ which contains some
  $\varepsilon$-neighborhood of the
  origin, $\cN_\varepsilon$.
  Suppose
 $f=(f[n])$ is a sequence
 of analytic functions $f[n]:\cV(n)\to \matnh$.
  If there is a formal power series
 $F$ such that for $X\in\cN_\varepsilon$ the series
$
  F(X)=\sum_m F^{(m)}(X)
$
 converges in norm to $f(X)$,
 then $f$ is a \cnca  $\cV\to\math$.
\end{proposition}

 The following lemma will be used in the proof 
 of Proposition \ref{prop:pow-nc-map}.

\begin{lemma}
 \label{lem:funnyconnected}
   Suppose $W$ is an open connected subset of
   a locally connected metric space $X$ and $o\in W$. Suppose
   $o\in W_1 \subset W_2 \subset \cdots$ is 
   a nested increasing sequence of open 
   subsets of $W$ and let $W_j^o$ denote the
   connected component of $W_j$ containing $o$.
   If $\cup W_j = W$, then $\cup W_j^o=W$. 
\end{lemma}

\begin{proof}
  Let $U=\cup W_j^o$.  If $U$ is a proper subset of $W$, 
  then $V=W\setminus U$ is neither empty nor open. 
  Hence, there is a $v\in V$ 
  such that $N_\delta(v)\cap U\ne \emptyset$ for every $\delta>0$. 
  Here $N_\delta(v)$ is the $\delta$ neighborhood
  of $o$. 

  There is an $N$ so that if $n\ge N$, then
  $v\in W_n$.  There is a $\delta>0$ such that
  $N_\delta(v)$ is connected, and $N_\delta(v)\subset W_n$
  for all $n\ge N$.  There is an $M$ so that 
  $N_\delta(v)\cap W_m^o\ne \emptyset$ for all $m\ge M$.
  In particular, since both $N_\delta(v)$ and $W_m^o$
  are connected, $N_\delta(v)\cup W_m^o$ is connected. 
  Hence, for $n$ large enough, 
  $N_\delta(v) \cup W_m^o$ is both connected
  and a subset of $W_m$.  This gives the contradiction
  $N_\delta(v)\subset W_m^o$.
\end{proof}

\begin{proof}[Proof of Proposition {\rm\ref{prop:pow-nc-map}}]
  For notational convenience, let $\cN=\cN_\varepsilon$. 
 For each $n$, the formal power series $F$ determines an analytic
 function $\cN(n) \to \matnh$ which
 agrees with $f[n]$ (on $\cN(n)$).  Moreover, if
 $X\in\cN(n)$ and $Y\in\cN(m)$, and $X\Gamma=\Gamma Y$,
 then $F(X)\Gamma =\Gamma F(Y)$. Hence
 $f[n](X)\Gamma = \Gamma f[m](Y)$.

 Fix $X\in\cV(n)$, $Y\in\cV(m)$, and suppose
  there exists  $\Gamma\neq 0$ such that
  $X\Gamma=\Gamma Y$. 
 For each positive  integer $j$ let 
\[
  \cW_j =\big\{(A,B)\in \cV(n)\times \cV(m):
   \begin{bmatrix} I & - \frac{1}{j} \Gamma \\ 0 & I \end{bmatrix}
   \begin{bmatrix} A & 0 \\ 0 & B \end{bmatrix}
   \begin{bmatrix} I & \frac{1}{j} \Gamma \\ 0 & I \end{bmatrix}
     \in \cV(n+m)\big\} \subset \cV(n)\oplus \cV(m).
\]
 Note that   $\cW_j$ is open since $\cV(n+m)$ is.
 Further,  $\cW_j \subset \cW_{j+1}$ for each $j;$ 
  for $j$ large enough, 
  $(0,0)\in \cW_j;$ and   $\cup  \cW_j = W:=\cV(n)\oplus \cV(m)$.
 By Lemma \ref{lem:funnyconnected}, $\cup W_j^o =W$, where
  $\cW_j^o$ is the connected component of 
  $\cW_j$ containing $(0,0)$.  Hence, $(X,Y)\in \cW_j^o$ 
  for large enough $j$ which we now fix. 
  Let $\cY\subset \cW_j$ be a connected neighborhood of $(0,0)$
  with $\cY\subset \cN(n)\oplus \cN(m)$.

 We have analytic functions
  $G,H:\cW_j^o\to M_{m+n}(\mathbb C^g)$ defined by
\[
 \begin{split}
     G(A,B)= &
    \begin{bmatrix} I & -\frac{1}{j}\Gamma \\ 0 & I \end{bmatrix}
    \begin{bmatrix} f(n)(A) & 0 \\ 0 & f(m)(B)\end{bmatrix}
   \begin{bmatrix} I & \frac{1}{j} \Gamma \\ 0 &  I \end{bmatrix} \\
   H(A,B) = & f(n+m)( \begin{bmatrix} I & - \frac{1}{j}
            \Gamma \\ 0 & I \end{bmatrix}
   \begin{bmatrix} A & 0 \\ 0 & B \end{bmatrix}
   \begin{bmatrix} I & \frac{1}{j} \Gamma \\ 0 & I \end{bmatrix}).
 \end{split}
\]
  For $(A,B)\in \cY$ we have $G(A,B)=H(A,B)$ from above.
  By analyticity and the connectedness of
  $\cW_j^o$, this shows $G(A,B)=H(A,B)$
  on $\cW_j^o$. 

  Since $(X,Y)\in \cW_j^o$ we obtain  
  the equality $G(X,Y)=H(X,Y),$ which  gives, using $X\Gamma -\Gamma Y=0$,
\[
  f\begin{bmatrix} X & 0 \\ 0 & Y \end{bmatrix}
  = \begin{bmatrix} f(X) & \frac{1}{j} (f(X)\Gamma - \Gamma f(Y)) \\ 
   0 & f(Y) \end{bmatrix}.
\]
  Thus $f(X)\Gamma-\Gamma f(Y)=0$ and 
  we conclude that $f$ respects intertwinings
  and hence is a \ncap
\end{proof}

 If $\cV$ is a noncommutative set, a \nca 
 $f:\cV\to \math$ is 
  {\bf uniformly bounded}\index{uniformly bounded, noncommutative function}\index{noncommutative function, uniformly bounded}
  provided
 there is a $C$ such that $\|f(X)\|\le C$ for every $n\in\N$ and $X\in\cV(n).$

\begin{proposition}
 \label{prop:rep-nc-map}
  If $f:\cN_\varepsilon \to\math$ is a \cnca
  then there is a formal power series
\beq
  F=\sum_{w\in\cFg} F_w w =\sum_{m=0}^\infty \sum_{|w|=m} F_w w
\eeq
 which converges on $\cN_\varepsilon$ and such that
 that $F(X)=f(X)$ for $X\in\cN_\varepsilon.$

 Moreover, if $f$ is uniformly bounded by $C$, then
 the power series converges uniformly in the sense
 that
for each $m$, $0\le r<1$,  and tuple $T=(T_1,\dots,T_g)$ of operators
 on Hilbert space satisfying $
  \sum T_j T_j^* \prec r^2 \varepsilon^2 I,$ 
we have
\[
  \Big\| \sum_{|w|=m} F_w \otimes T^w \Big\| \le C r^m.
\]
 In particular, $\|F_w\|\le\frac{C}{\varepsilon^n}$ for each word $w$
 of length $n$. 
 \end{proposition}

\begin{remark}\rm
  Taking advantage of polynomial identities for $M_n(\C)$,
  the article \cite{Vo2} gives an example of a formal
  power series $G$ which converges for every tuple
  $X$ of matrices, but has $0$ radius of convergence
  in the sense that for every $r>0$ there exists
  a tuple of operators $X=(X_1,\cdots,X_g)$ with
  $\sum X_j^*X_j < r^2$ for which $G(X)$ fails
  to converge. 
\end{remark}

\subsection{The Fock space}
 \label{appendix}
   We now start proving Proposition
  \ref{prop:rep-nc-map}.

\subsection{The creation operators}
 \label{sec:creat}
 The {\bf noncommutative Fock space}\index{Fock space},
 denoted $\cFgtwo,$ is the Hilbert space with
 orthonormal basis $\cFg$.  For $1\le j\le g$,
 the operators $S_j:\cFgtwo\to \cFgtwo$ determined
 by $S_j w=x_j w$ for words $w\in\cFg$ are
 called the {\bf creation operators}\index{creation operators}.
 It is readily checked that each $S_j$ is an isometry and
\[
  I-P_0 =\sum S_j S_j^*,
\]
 where $P_0$ is the projection onto the one-dimensional
 subspace of $\cFgtwo$ spanned by the empty word $\emptyset$.
 As is well known \cite{Fr,Po0},
the creation operators serve as a universal model
 for row contractions. We
  state a precise
 version of this result suitable for our purposes
 as Proposition \ref{thm:dilate} below.

 Fix a positive integer $\ell$.
 A tuple $X\in \matng$ is {\it nilpotent of order $\ell+1$}
 if $X^{w}=0$ for any word $w$ of length $|w|>\ell$.
 Let $\cP_\ell$ denote
 the subspace of $\cFgtwo$ spanned by words of length
 at most $\ell$; $\cP_\ell$ has dimension
 $$\sigma(\ell)=\sum_{j=0}^\ell g^j.$$  Let $V_\ell:\cP_\ell\to \cFgtwo$
  denote the inclusion mapping and let
\[
  V_\ell^* S V_\ell=V_\ell^*(S_1,\dots,S_g)V_\ell
   =(V_\ell^* S_1 V_\ell,\dots, V_\ell^* S_gV_\ell).
\]
  As is easily verified,
 the subspace $\cP_\ell$ is invariant for each $S_j^*$
 (and thus semi-invariant (i.e., the orthogonal difference of two invariant
  subspaces) for $S_j$).  Hence, for a   polynomial $p\in\bCx$,
\[
  p(V_\ell^*SV_\ell)=V_\ell^*p(S)V_\ell.
\]
  In particular,
\[
  \sum_j (V_\ell^* S_jV_\ell) ( V_\ell^* S_j^* V_\ell)
    \prec V_\ell^* \sum_j S_j S_j^* V_\ell = V_\ell^* P_0 V_\ell.
\]
 Hence, if $|z|<\varepsilon$, then $V_\ell^* zS V_\ell$ is
 in $\cN_\varepsilon$, the $\varepsilon$-neighborhood of $0$
  in $\matg$.

 The following is a well known 
 algebraic version of a classical dilation theorem.
 The proof here follows along the lines of the 
 de Branges-Rovnyak construction of the coisometric dilation of
 a contraction operator on Hilbert space \cite{RR}.

\begin{proposition}
 \label{thm:dilate}
 Fix a positive integer $\ell$ and let $T=V_\ell^* S V_\ell$.
  If  $X\in\matng$ is nilpotent of order $\ell$ and if
$
  \sum X_j X_j^*  \prec  r^2 I_n
$
  then there is an isometry $V:\mathbb  C^n \to \mathbb C^n\otimes \cP_\ell$
  such that
$
   V X_j^* = r (I\otimes T_j^*) V,
$
  where $I$ is the identity on $\mathbb C^n$.
\end{proposition}

\begin{proof}
We give a
de Branges-Rovnyak style proof 
 By scaling, assume that $r=1$.
  Let
\[
  R=\sum X_j X_j^*.
\]
 Thus, by hypothesis  $0\preceq R \prec I$.
  Let
\[
  D=(I-\sum T_jT_j^*)^\frac12.
\]
 The matrix $D$ is known as the {\bf defect} and, by hypothesis,  is
 strictly positive definite.  Moreover,
\begin{equation}
 \label{eq:isometry}
  \sum_{|w|\le \ell}  X^w D D (X^w)^* = I- \sum_{|w|=\ell+1}X^w (X^w)^* = I.
\end{equation}

 Define $V$ by
\[
  V\gamma = \sum_w D(T^w)^* \gamma \otimes w.
\]
 The equality of equation \eqref{eq:isometry}
 shows that $V$ is an isometry. Finally
\[
 \begin{split}
  VX_j^* \gamma & =  \sum_{|w|\le \ell-1} D(X^w)^* X_j^* \gamma\otimes w
    =  \sum_{|w|\le \ell-1} D (X^{x_jw})^* \gamma\otimes w \\
    & =  T_j^* \sum_{|w|\le \ell-1} D (X^{x_jw})^*\gamma \otimes x_j w \\
    & =  S_j^* \big(D\gamma +\sum_k \sum_{|w|\le \ell-1} D(T^{x_kw})^* \gamma
                \otimes x_k w  \big) \\
    &    =  S_j^* V\gamma. \qedhere
 \end{split}
\]
\end{proof}

\subsection{Creation operators meet \ncas}
 In this section we determine formulas for the coefficients $F_w$
 of Proposition \ref{prop:rep-nc-map} of
 the power series expansion of $f$ in terms of the
 creation operators  $S_j$. Formulas
 for the $F_w$ are also given in \cite[Section 13]{Vo2} and in \cite{KVV},
 where they are obtained by clever substitutions and have nice properties.
 Our formulas 
in terms of the familiar creation operators and
 related algebra  provide a slightly different perspective and  impose an organization which might
 prove interesting.

\begin{lemma}
 \label{lem:fF}
  Fix a positive integer $\ell$ and let $T= V_\ell^* SV_\ell$
  as before.
  If $f:\matg\to \math$ is a \ncac
  then  there exists, for
  each word $w$ of length at most $\ell$, a vector
  $F_w\in\mathbb C^{\h}$ such that
 \[
   f(T)=\sum_{|w|\le \ell} F_w \otimes T^w.
 \]
\end{lemma}

 Given $u,w\in\cFg$, we say {\bf $u$ divides $w$ (on the right)},
 \index{divides on the right} denoted $u | w$,
 if there is a $v\in\cFg$ such that $w=uv$.

\begin{proof}
  Fix a word $w$ of length at most $\ell$. Define $F_w\in\mathbb C^{\h}$
  by
\[
  \langle F_w,\vy\rangle
     = \langle \emptyset, f(T)^* \vy\otimes w\rangle,
\quad \vy\in\mathbb C^{\h}.
\]

  Given a word $u\in\cP_\ell$ of length $k$,
  let $R_u$ denote the operator of {\em right} multiplication
  by $u$ on $\cP_\ell$. Thus, $R_u$ is determined by
  $R_u v= vu$  if $v\in \cFg$ has length
  at most $\ell-k$, and $R_u v=0$ otherwise.
  Routine calculations show
 \[
    T_j R_u = R_u T_j.
 \]
  Hence, for the \nca $f$, $
    f(T) R_u = R_u f(T).
$
   Thus, for words $u,v$ of length at most $\ell$
   and $\vy\in\mathbb C^{\h}$,
 \[
   \langle u, f(T)^* \vy \otimes v \rangle
   = \langle R_u \emptyset, f(T)^* \vy \otimes v\rangle
   = \langle \emptyset, f(T)^* \vy \otimes R_u^* v \rangle.
 \]
  It follows that
\beq
 \label{eq:fF1}
  \langle f(T)^* \vy\otimes v,u\rangle
   =\begin{cases}
         \langle \vy, F_\alpha\rangle & \mbox{ if } v=\alpha u \\
              0 & \mbox{ otherwise}. 
     \end{cases}
 \eeq

  On the other hand, if $v=wu$, then $(T^w)^* v =u$ and 
   otherwise, $(T^w)^* v$ is orthogonal to $u$. 
  Thus,
\begin{equation}
 \label{eq:fF2}
   \langle \sum F_w^* \otimes (T^*)^w \vy\otimes v, u\rangle  = 
     \begin{cases} F_w^* \vy & \mbox{ if } v=wu \\
           0 & \mbox{ otherwise.}
   \end{cases}
\end{equation}
  Comparing equations \eqref{eq:fF1} and \eqref{eq:fF2} completes
  the proof.
\end{proof}

\begin{lemma}
 \label{lem:creation-vs-ncmap}
  Fix a positive integer $\ell$ and, as
  in Proposition {\rm\ref{thm:dilate}}, let
  $T=V_\ell^* SV_\ell$ act on $\cP_\ell$.
  Suppose $V:\mathbb C^n\to \mathbb C^n \otimes \cP_\ell$
  is an isometry and $X\in \matng$. If
  $f:\matg\to \math$ is a \nca
  and $V X^*  = (I\otimes T^*) V,$ then
 \[
   f(X)=V^* \big(I\otimes f(T)\big) V.
 \]
\end{lemma}

\begin{proof}
  Taking adjoints gives
 $
   X V^* = V^* (I\otimes T).
 $
  From the definition of a \ncacomma
 \[
    f(X)V^* = V^* (I\otimes f(T)).
 \]
  Applying $V$ on the right and using the fact that $V$ is
  an isometry completes the proof.
\end{proof}

\begin{remark}\rm
  Iterating the intertwining relation $VX^*=(I\otimes T^*) V$, it
  follows that, $V(X^w)^* = (I\otimes (T^w)^*)V$.  In particular, 
  if $F$ is formal power series, then $F(X^*)V=V F(I\otimes T^*)$.
\end{remark}

 A \nca $f:\matg\to\math$ is
 {\bf homogeneous of degree $\ell$} if for
  all $X\in\matg$  and $z\in\mathbb C$,
 $
  f(zX)=z^\ell f(X).
 $

\begin{lemma}
 \label{lem:degl}
  Suppose $f:\matg \to \math$ is a \ncap
  If $f$ is continuous and
homogeneous of degree $\ell$, then there
  exists, for each word $w$ of length $\ell$,
  a vector $F_w\in\mathbb C^{\h}$ such that
 \[
   f(X)=\sum_{|w|=\ell} F_w \otimes X^w \quad \text{for all }
X\in \matg.
 \]
\end{lemma}

\begin{proof}
  Write $T=V_\ell^* S V_\ell$.
  Let $n$ and $X\in\matng$ be given and assume 
 $
    \sum X_j X_j^* \prec I.
 $
  Let $J$ denote the nilpotent
  Jordan block of size $(\ell+1) \times (\ell+1)$.
  Thus the entries of $J$ are zero, except for the $\ell$
  entries along the first super diagonal which are all $1$.
  Let $Y=X\otimes J$. Then $Y$ is nilpotent of order $\ell + 1$ and 
 $
   \sum Y_j Y_j^* \prec I.
 $
   By Proposition \ref{thm:dilate},  there is an isometry
  $V:\mathbb C^n\otimes \mathbb C^{\ell+1} \to
     (\mathbb C^n\otimes \mathbb C^{\ell+1})\otimes \cP_\ell$
  such that
 \[
   V Y^* = (I\otimes T^*) V.
 \]
   By Theorem \ref{lem:creation-vs-ncmap},
 $
    f(Y) = V^* (I\otimes f(T))V.
 $
  From Lemma \ref{lem:fF} there exists, for words $w$
  of length at most $\ell$, vectors $F_w\in\mathbb C^{\h}$ such that
 $
   f(T)=\sum_{|w|\le \ell} F_w\otimes T^w.
 $
  Because $f$ is a free map, $f(I\otimes T)=I\otimes f(T)$. 
   Hence,
 \[
    f(Y) =   \sum_{|w|\le \ell} F_w\otimes V^* (I\otimes T^w) V
         =  \sum_{|w|\le \ell}  F_w \otimes Y^w
         =  \sum_{m=0}^\ell \big(\sum_{|w|=m} F_w \otimes X^w \big) \otimes J^m
 \]
   Replacing $X$ by $zX$ and using the homogeneity of $f$ gives,
 \[
   z^\ell f(Y) = \sum_{m=0}^\ell \big(\sum_{|w|=m} F_w \otimes X^w \big) \otimes z^m J^m
 \]
   It follows that
 \begin{equation}
  \label{eq:meet1}
   f(Y) = \big(\sum_{|w|=\ell} F_w \otimes X_w\big) \otimes J^\ell.
 \end{equation}

  Next suppose that $E=D+J$, where $D$ is diagonal with distinct
  entries on the diagonal.  Thus there exists an invertible matrix
  $Z$ such that $ZE=DZ$.  Because $f$ is a \ncacomma
 $
   f(X\otimes D) = \oplus f(d_j X),
 $
  where $d_j$ is the $j$-th diagonal entry of $D$. Because
 of the homogeneity of $f$,
 \[
   f(X\otimes D) = \oplus d_j^\ell X = f(X)\otimes D^\ell.
 \]
   Hence,
 \[
    f(X\otimes E)  =  (I\otimes Z^{-1}) f(X\otimes D) (I\otimes Z)
       = (I\otimes Z^{-1}) f(X)\otimes D^\ell (I\otimes Z) 
        = f(X) \otimes E^\ell.
 \]
   Choosing a sequence of $D$'s which converge to $0$, so
   that the corresponding $E$'s converge to $J$, and using
   the continuity of $f$ yields $   f(Y) = f(X)\otimes J^\ell.$
   A comparison with \eqref{eq:meet1} 
   proves the lemma.
\end{proof}

\def\eps{\varepsilon}

\subsection{The proof of Proposition {\rm\ref{prop:rep-nc-map}}}
Let $f:\cN_\varepsilon\to \math$ be a \cncap
  Given $X\in\matng$, there is a disc
 $D_X=\{z\in\mathbb C: |z|<r_X\}$ such that $zX\in\cN_\varepsilon$
 for $z\in D_X$.  By analyticity of $f$, the function
 $D_X\ni z\mapsto f(zX)$ is analytic (with values in $\matnh$)
 and thus has a power series expansion,
\[
  f(zX)=\sum_m A_m z^m.
\]
 These $A_m=A_m(X)$ are uniquely determined by $X$ and hence there exist
 functions $\fhmn(X)=A_m(X)$ mapping $\matng$ to $\matnh$.
 In particular, if $X\in\cN_\varepsilon(n)$, then
\begin{equation}
 \label{eq:pow}
   f(X) = \sum \fhmn (X).
\end{equation}

\begin{lemma}
 For each $m$, the sequence $(\fhmn)_n$ is a continuous
 \nca $\matg\to\math$.
  Moreover, $\fhm$ is homogeneous of degree $m$.
\end{lemma}

 \begin{proof}
   Suppose $X,Y\in\matg$ and $X\Gamma = \Gamma Y$. For
  $z\in D_X\cap D_Y$,
 \[
    \sum \fhm(X)\Gamma  z^m
       =  f(zX)\Gamma 
       =  \Gamma f(zY) 
       =  \sum  \Gamma \fhm(Y) z^m.
 \]
   Thus  $\fhm(X)\Gamma =\Gamma \fhm(Y)$ for each $m$ and
  thus each $\fhm$ is a free map.
   Since $\fn$ is continuous, so is $\fhmn$ for each $n$.

   Finally, given $X$ and $w\in\mathbb C$, for $z$ of sufficiently
   small modulus,
 \[
   \sum \fhm(wX) z^m =  f(z(wX))
     =  f(zw X)
     =  \sum \fhm(X)w^m z^m.
 \]
   Thus $\fhm(wX)=w^m \fhm(X)$.
 \end{proof}

 Returning to the proof of Proposition \ref{prop:rep-nc-map},
 for each $m$, let $F_w$ for a word $w$ with $|w|=m$, denote the
 coefficients produced by Lemma \ref{lem:degl}
 so that
\[
  \fhm(X) =\sum_{|w|=m} F_w \otimes X^w.
\]
  Substituting into equation \eqref{eq:pow}
  completes the proof of
  the first part of the Proposition \ref{prop:rep-nc-map}.

  Now suppose that $f$ is uniformly bounded by $C$
  on $\cN$. If $X\in\cN,$ then
\[
  C\ge  \Big\| \frac{1}{2\pi} \int f(\exp(it)X)\exp(-imt)\, dt \big\|  
    =  \| \fhm(X)\|.
\]
  In particular, if $0<r<1$, then $\|\fhm(rX)\|\le r^m C$.

  Let $T=V_m^* S V_m$ as in Subsection \ref{sec:creat}.
   In particular, if $\delta <\varepsilon$, then $\delta T\in \cN$ and
   thus 
\[
  C^2  \ge \| \fhm(\delta T)\emptyset \|^2 
        = \delta^{2m} \sum_{|v| = m}  \|F_v \|^2.
\] 
  Thus, $\|F_v\|\le \frac{C}{\delta^m}$ for all $0<\delta <\varepsilon$
  and words $v$ of length $m$ and the last statement of 
  Proposition \ref{prop:rep-nc-map} follows. 
 \qed

\section{Proper \ncas}
 \label{sec:main-result}
   Given noncommutative domains $\cU$ and $\cV$ in
  $\matg$ and $\math$ respectively, a
  \nca $f:\cU\to\cV$ is {\bf proper} \index{proper}
  if each $\fn:\cU(n)\to \cV(n)$ is proper in the
  sense that if $K\subset \cV(n)$ is compact, then
  $f^{-1}(K)$ is compact.
  In particular,  for all $n$,
  if $(z_j)$ is a sequence in $\cU(n)$
  and $z_j\to\partial\cU(n)$, then
  $f(z_j)\to\partial\cV(n)$.
 In the case $g=\h$ and both $f$ and $f^{-1}$ are (proper)
 \cncas
 we say $f$ is 
 a 
{\bf free biholomorphism}.

\subsection{Proper implies injective}\label{sec:properInj}
  The following theorem was established in \cite[Theorem 3.1]{HKM3}.
We will not give the proof here but instead record a few
corollaries below.

\begin{theorem}
\label{thm:oneone}
  Let $\cU$ and $\cV$ be noncommutative domains containing $0$
  in $\matg$ and $\math$, respectively and
  suppose $f:\cU\to \cV$ is a \ncap
\begin{enumerate}[\rm (1)]
\item\label{it:1to1}
If $f$ is proper, then it is one-to-one,
  and $f^{-1}:f(\cU)\to \cU$ is a \ncap
\item\label{it:1to1ugly}
  If, for each $n$ and $Z\in\matnh$,
  the set $\fn^{-1}(\{Z\})$ has compact closure in $\cU$, 
  then $f$ is one-to-one
  and moreover, $f^{-1}:f(\cU)\to \cU$ is a \ncap
\item\label{it:xto1}
If $g=\h$ and $f:\cU\to\cV$ is
 proper and  continuous, then $f$ is biholomorphic.
\end{enumerate}
\end{theorem}

\begin{corollary}
 \label{cor:bianalytic-free}
  Suppose $\cU$ and $\cV$ are noncommutative domains in
 $\matg$. If $f:\cU\to \cV$ 
  is a free map and if each $f[n]$ is biholomorphic, 
  then $f$ is a free biholomorphism.
\end{corollary}

\begin{proof}
 Since each $\fn$ is biholomorphic, each $\fn$ is proper. Thus $f$ is
 proper. Since also $f$ is a \ncacomma by Theorem
  \ref{thm:oneone}(\ref{it:xto1})  $f$ is
 a \fbap
\end{proof}

\begin{corollary}
 \label{cor:fprime}
  Let $\cU\subset\matg$ and $\cV\subset\math$ be  noncommutative
  domains.
  If $f:\cU\to\cV$ is a
  proper \cnca
 and if $X\in\cU(n)$, then
  $f^\prime(X):\matng\to \matnh$ is one-to-one.
  In particular, if $g=\h$, then $f^\prime(X)$
  is a vector space isomorphism.
\end{corollary}

\begin{proof}
  Suppose $f^\prime(X)[H]=0$.
 We scale $H$ so that $\begin{bmatrix} X & H \\ 0 & X\end{bmatrix} \in\cU$.
 From Proposition \ref{prop:continuous-analytic},
\[
  f\begin{bmatrix} X & H \\ 0 & X\end{bmatrix}
   = \begin{bmatrix} f(X) & f^\prime(X)[H] \\ 0 & f(X)\end{bmatrix}=
  \begin{bmatrix} f(X) & 0 \\ 0 & f(X)\end{bmatrix}
   = f\begin{bmatrix} X & 0 \\ 0 & X\end{bmatrix}.
\]
By the injectivity of $f$ established in Theorem \ref{thm:oneone},
$H=0$.
\end{proof}

\begin{remark}\rm
Let us note that Theorem \ref{thm:oneone} is sharp
as explained in \cite[\S 3.1]{HKM3}:
absent more conditions on the noncommutative
domains $\cU$ and $\cV$, nothing beyond  free biholomorphic
can be concluded about $f$.
\end{remark}

A natural condition on a noncommutative domain $\cU$, 
which we shall consider in \S\ref{sec:ku},
is circularity.
However, we first proceed to give some
free analogs of well-known results from several complex
variables.

\section{Several analogs to classical theorems}
\label{sec:analogs}
The conclusion of Theorem \ref{thm:oneone} is sufficiently
strong that most would say that it does not have a classical analog.
Combining it with 
 classical
several complex variable theorems yields \cnca 
analogs. Indeed,
hypotheses for these analytic \nca results are weaker than their
classical analogs would suggest.

\subsection{A free Carath\'eodory-Cartan-Kaup-Wu (CCKW) Theorem}
\label{sec:onto2}
 The commutative Carath\'eodory-Cartan-Kaup-Wu (CCKW) Theorem
 \cite[Theorem 11.3.1]{Krantz} says that
 if $f$ is an analytic self-map of a bounded domain
 in $\mathbb C^g$ which fixes
 a point $P$, then the eigenvalues of $f^\prime(P)$
 have modulus at most one. Conversely, if the eigenvalues
 all have  modulus one, then $f$ is in fact an automorphism;  and
 further if $f^\prime(P)=I$, then $f$ is the identity.
 The CCKW Theorem together with Corollary \ref{cor:bianalytic-free}
 yields Corollary \ref{cor:cckw1} below.
 We note
 that Theorem \ref{thm:oneone} can also be thought of
 as a noncommutative CCKW theorem in that
it concludes, like the CCKW Theorem does,  
 that a  map $f$ is biholomorphic, 
 but under the (rather different) assumption that $f$ is proper.

Most of the
proofs in this section are skipped and can be found in \cite[\S 4]{HKM3}.

\begin{corollary}[{\protect\cite[Corollary 4.1]{HKM3}}]
\label{cor:cckw1}
 Let $\cD$ be a given bounded
noncommutative domain
 which contains $0$.
 Suppose
 $f:\cD \to \cD$ is
 an \cncap
   Let $\phi$ denote the
 mapping $\fone:\cD(1)\to\cD(1)$ and assume $\phi(0)=0.$
\begin{enumerate}[\rm (1)]
 \item  If  all the eigenvalues of
 $\phi^\prime(0)$ have modulus one,
   then $f$ is a   free biholomorphism; and
 \item  if $\phi^\prime(0)=I$, then $f$ is the identity.
\end{enumerate}
\end{corollary}

Note a classical biholomorphic function $f$ is
completely determined by its value and differential
at a point
 (cf.~a remark after \cite[Theorem 11.3.1]{Krantz}).
Much the same is true for \cncas and for the same reason.

\begin{proposition}
 \label{prop:derivative0}
   Suppose $\cU,\cV\subset \matg$ are noncommutative domains, $\cU$
   is bounded, both contain $0,$
 and $f,g:\cU\to \cV$ are proper \cncasp
 If $f(0)=g(0)$ and $f^\prime(0)=g^\prime(0)$, then $f=g$.
\end{proposition}

\begin{proof}
 By Theorem \ref{thm:oneone} both $f$ and $g$ are
 \fbasp
Thus $h=f\circ g^{-1}:\cU\to \cU$
 is a \fba fixing $0$ with $h[1]^\prime(0)=I$.
 Thus, by Corollary \ref{cor:cckw1}, $h$ is the identity.
 Consequently $f=g$.
\end{proof}

\subsection{Circular domains}
  A subset $S$ of a complex vector space is
  {\bf circular} if $\exp(it) s\in S$  whenever
  $s\in S$ and $t\in\R$.
  A noncommutative domain $\cU$ is circular
  if each $\cU(n)$ is circular.\index{circular domain}

 Compare the following theorem  to
 its commutative counterpart \cite[Theorem  11.1.2]{Krantz}
 where the domains $\cU$ and $\cV$ are the same.

\begin{theorem}
\label{thm:circLin}
  Let $\cU$ and $\cV$ be bounded noncommutative domains
  in $\matg$ and $\math$, respectively, both
  of which contain $0$.
  Suppose $f:\cU\to \cV$ is a
  proper \cnca
  with $f(0)=0$.
  If $\cU$ and  the range
 $\cR:= f(\cU)$ of $f$
 are circular,  then $f$ is linear.
\end{theorem}

 The domain $\cU=(\cU(n))$ is {\bf weakly convex}
(a stronger notion of convex for a noncommutative domain appears
later)
if each $\cU(n)$ is a convex set.
 Recall a set $C\subseteq\C^g$ is convex, if for every
$X,Y\in C$, $\frac{X+Y}2\in C$.

\begin{corollary}
  Let $\cU$ and $\cV$ be bounded noncommutative domains
  in $\matg$  both
  of which contain $0$.
  Suppose $f:\cU\to \cV$ is a
  proper \cnca  with $f(0)=0$.
  If both $\cU$ and $\cV$ are circular and if one is weakly convex,
  then so is the other.
\end{corollary}

This corollary is an immediate
consequence of Theorem \ref{thm:circLin} and the fact
(see Theorem \ref{thm:oneone}(\ref{it:xto1})) that
$f$ is onto $\cV$. 

 We admit the hypothesis that the range $\cR= f(\cU)$ of $f$
 in Theorem \ref{thm:circLin}
 is circular seems pretty contrived when the domains
 $\cU$ and $\cV$ have a different number of variables.
 On the other hand if they  have the same number of variables
 it is the same as $\cV$ being circular since by
 Theorem \ref{thm:oneone}, $f$ is onto.

\begin{proof}[Proof of Theorem {\rm\ref{thm:circLin}}]
 Because $f$ is a proper \nca it is injective
 and its inverse (defined on $\cR$) is a \nca
 by Theorem \ref{thm:oneone}.  Moreover, using the
 analyticity of $f$,  its derivative
 is pointwise injective by Corollary \ref{cor:fprime}.
 It follows that each $\fn:\cU(n)\to \matnh$
 is an embedding \cite[p.~17]{GP}.
 Thus, each $\fn$  is a homeomorphism onto its range and
 its  inverse $\fn^{-1}=f^{-1}[n]$ is continuous.

 Define $F: \cU \to \cU$ by
\beq\label{eq:defF}
  F(x):=  f^{-1}\big( \exp(- i\theta) f(\e x) \big)
\eeq
 This function respects direct sums and similarities,
 since it is the composition of maps which do.
  Moreover, it is continuous by the discussion above.
 Thus $F$ is a \cncap

 Using the relation $\e f(F(x))= f(\e)$
 we find $\e f^\prime(F(0))F^\prime(0)=f^\prime(0)$.
 Since $f^\prime(0)$ is injective, $\e F^\prime(0)=I.$
 It follows from Corollary \ref{cor:cckw1}(2)
 that $F(x)=\e x$ and thus,
 by \eqref{eq:defF},
 $f(\e x)=\e f(x)$. Since this holds for every
  $\theta$, it follows that $f$ is linear.
\end{proof}

If $f$ is not assumed to map $0$ to $0$ (but instead fixes
some other point), then a
proper self-map 
need not be linear.
This follows from the  example  we discuss in
\S\ref{sec:exYes}.

\section{A free Braun-Kaup-Upmeier (BKU) Theorem}\label{sec:ku}

 Noncommutative domains $\cU$ and $\cV$ are {\bf freely biholomorphic}
 if there exists a free biholomorphism $f:\cU\to \cV$. 
 In this section we show how a theorem of 
 Braun-Kaup-Upmeier \cite{BKU,KU} can be used
 to show that  bounded circular noncommutative domains 
  that are freely biholomorphic are (freely) linearly biholomorphic.

\begin{defn}
  Given a domain $D\subset \mathbb C^g,$ let 
  $\Aut(D)$ denote the group of all biholomorphic
  maps from $D$ to $D$. Note that $D$
  is circular if and only if $\Aut(D)$
  contains all rotations; i.e., all maps of the form $z \mapsto \e z$
for $\theta\in\R$.
\end{defn}

 Let $\cD=(\cD(n))$ be a circular noncommutative domain.  
 Thus each $\cD(n)$ is open, connected,
 contains $0$ and is invariant under rotations.
 The set
 $\cD(1)\subset\mathbb C^g$ is in particular a circular domain in
 the classical
 sense and moreover $\Aut(\cD(1))$ contains all rotations.

\begin{theorem}[A free BKU Theorem]
 \label{thm:nc-circular}
   Suppose $\cU$ 
   and $\cD$ are 
   bounded, circular noncommutative domains which contain noncommutative neighborhoods of $0$. 
  If $\cU$ and $\cD$ are freely biholomorphic, then there is a 
  linear $($free$)$ biholomorphism $\lambda:\cD\to \cU$.
\end{theorem}

 A noncommutative domain $\cD$ containing $0$ is {\bf convex} 
 if it is closed with respect to conjugation by contractions; i.e.,
 if $X\in\cD(n)$ and $C$ is a $m\times n$ contraction, then
\[
  CXC^* = (CX_1C^*,CX_2C^*,\dots,CX_gC^*) \in \cD(m).
\]
  It is not hard to see, using the fact that noncommutative domains
   are also closed with respect to direct sums, that 
  each $\cD(n)$ is itself convex.  In the case
  that $\cD$ is semialgebraic, then in fact an easy
  argument shows that the converse is true: if each $\cD(n)$
  is convex ($\cD$ is weakly convex), then $\cD$ is convex. (What is used
  here
  is that the domain is closed with respect to restrictions
  to reducing subspaces.)  In fact, in 
  the case that $\cD$ is semialgebraic and convex, it
  is equivalent to being an LMI,
cf.~\cite{HM} for precise statements and proofs;
the topic is also addressed briefly in \S\ref{sec:convexIsLMI} below.
As an important corollary of Theorem \ref{thm:nc-circular}, 
we have the following nonconvexification 
result.

\begin{corollary}
 \label{cor:bku-for-lmi}
   Suppose $\cU$ is a bounded circular noncommutative domain which contains
   a noncommutative neighborhood of $0$. 

\ben[\rm(1)]
\item
   If $\cU$ is freely biholomorphic to
   a bounded circular weakly convex noncommutative domain that
   contains a noncommutative neighborhood of $0$, then $\cU$ is
   itself convex.
\item
   If $\cU$ is freely biholomorphic to a bounded circular LMI domain,
  then $\cU$ is itself an LMI domain.
\een
\end{corollary}

\begin{proof}
 It is not hard to see that an LMI domain does in fact contain a
 noncommutative neighborhood of the origin.  Thus, both statements of
 the corollary follow immediately from the theorem. 
\end{proof}

 Note that the corollary is in the free spirit of the main
  result of \cite{KU}.

\begin{remark}\rm
A main motivation for our line of research was investigating 
\emph{changes of variables} with an emphasis on achieving
convexity.
  Anticipating that the main result from \cite{HM} applies
  in the present context (see also \S\ref{sec:convexIsLMI}),
   if $\cD$ is a convex, bounded, noncommutative 
    semialgebraic set 
   then it is an LMI domain.  In this way, the hypothesis in the last
  statement of the corollary could be rephrased as:
   if $\cU$ is freely biholomorphic to a bounded circular convex
   noncommutative semialgebraic set, then $\cU$ is itself an LMI domain.  
   In the context of \S\ref{subsec:motiv},  the conclusion 
   is that in this circumstance domains biholomorphic 
   to bounded,  convex, circular basic semialgebraic sets
   are already in fact determined by an LMI. Hence  there no 
   nontrivial changes of variables in this setting.
\end{remark}

 For the reader's convenience we include here
 the version of \cite[Theorem 1.7]{BKU} 
 needed in the proof of
 Theorem \ref{thm:nc-circular}. Namely,
 the case in which  the ambient domain is $\mathbb C^g.$
  Closed here means closed in the topology of 
  uniform convergence on compact subsets.
  A bounded domain $D\subset \mathbb C^g$ is
  symmetric if for each $z\in D$ there is an involutive 
  $\varphi \in \Aut(D)$ such that $z$ 
  is an isolated fixed point of $\varphi$ \cite{Hg}. 

\begin{theorem}[\cite{BKU}]
 \label{thm:BKU-really}
 Suppose $S\subset \mathbb C^g$ is a bounded circular domain
 and $G \subset \Aut(S)$ is a closed subgroup of $\Aut(S)$
 which contains all rotations.
 Then
\ben[\rm (1)]
\item
 there is a
 closed $(\mathbb C$-linear$)$ subspace $M$ of $\C^g$ such that
 $A := S \cap M =G(0)$ is the orbit of the origin.
\item
 $A$ is a bounded symmetric 
 domain in $M$ and coincides with
\[
  \{z \in S : G(z) \text{ is a closed complex submanifold of } S \}.
\]
\een
 In particular  two
 bounded circular domains are biholomorphic
 if and only if they are linearly biholomorphic.
\end{theorem}

 We record the following 
 simple lemma  before turning to the
 proof of Theorem \ref{thm:nc-circular}.

\begin{lemma}
 \label{lem:scalar-cckw-app}
  Let $D\subset \mathbb C^g$ be a bounded domain
  and suppose $(\varphi_j)$ is a sequence from
  $\Aut(D)$ which converges uniformly on compact
  subsets of $D$ to $\varphi \in \Aut(D).$
\ben[\rm(1)]
\item
$\varphi_j^{-1}(0)$ converges to $\varphi^{-1}(0)$;
\item
If the
  sequence $(\varphi_j^{-1})$ converges uniformly
  on compact subsets of $D$ to $\psi$, then
  $\psi=\varphi^{-1}$.
\een
\end{lemma}

\begin{proof}
(1)
  Let $\varepsilon >0$ be given.  The sequence $(\varphi_j^{-1})$
  is a uniformly bounded sequence and is thus locally 
   equicontinuous.  Thus, there is a $\delta>0$ such that
  if $\|y-0\|<\delta$, then $\|\varphi_j^{-1}(y)-\varphi_j^{-1}(0)\|<\varepsilon$.
   On the other hand, $(\varphi_j(\varphi^{-1}(0)))_j$ converges to
   $0$, so for large enough $j$, $\|\varphi_j(\varphi^{-1}(0))-0\|<\delta.$
   With $y=\varphi_j(\varphi^{-1}(0))$, it follows that
   $\| \varphi_j(\varphi^{-1}(0)) - 0 \|< \varepsilon$.

(2)
  Let $f=\varphi(\psi)$. From the first part of the lemma, 
  $\psi(0)=\varphi^{-1}(0)$ and hence $f(0)=0$.  
  Moreover, $f^\prime(0)=\varphi^\prime(\psi(0)) \psi^\prime(0)$.
  Now $\varphi_j^\prime$ converges uniformly on compact sets
  to $\varphi^\prime.$ Since also $\varphi_j^\prime(\psi(0))$
  converges to $\varphi^\prime(\psi(0))$, it follows that
  $\varphi_j^\prime(\varphi_j^{-1}(0))$ converges to $\varphi^\prime(\psi(0))$.
  On the other hand,
  $I=\varphi_j^\prime(\varphi_j^{-1}(0)) (\varphi_j^{-1})^\prime(0).$
  Thus, $f^\prime(0)=I$ and we conclude, from 
    a theorem of Carath\'eodory-Cartan-Kaup-Wu (see Corollary
  \ref{cor:cckw1}), that $f$ is the identity.
  Since $\varphi$ has an (nc) inverse, $\varphi^{-1}=\psi$. 
\end{proof}

\begin{defn}
 Let $\Aut_{\rm nc}(\cD)$ denote the free automorphism group of the
 noncommutative domain $\cD$. Thus $\Aut_{\rm nc}(\cD)$ is the set 
 of all free biholomorphisms $f:\cD\to \cD$.  It is evidently a group
 under composition.  Note  that $\cD$ is circular implies 
 $\Aut_{\rm nc}(\cD)$ contains all
  rotations.   Given $g\in \Aut_{\rm nc}(\cD)$, let $\tilde{g}\in\Aut(\cD(1))$ 
  denote its
 commutative collapse; i.e., $\tilde{g}=g[1]$.
\end{defn}

\begin{lemma}
 \label{lem:unique-extend}
 Suppose $\cD$ is a bounded 
  noncommutative domain containing $0$. Assume
  $f,h\in \Aut_{\rm nc}(\cD)$ satisfy $\tilde{f}=\tilde{h}$. Then $f=h$.
\end{lemma}

\begin{proof}
  Note that  $F=h^{-1}\circ f\in\Aut_{\rm nc}(\cD).$
  Further, since $\tilde{F}=x$ (the identity), 
  $F$ maps $0$ to $0$ and $\tilde{F}^\prime(0)=I$. Thus,
  by Corollary \ref{cor:cckw1}, $F=x$ and therefore
  $h=f$. 
\end{proof}

\begin{lemma}
 \label{lem:sequential}
    Suppose $\cD$ is a noncommutative  domain which contains
   a noncommutative neighborhood of $0$, and $\cU$ is a bounded noncommutative domain.
    If $f_m:\cD\to \cU$ is a sequence of \cncasc then
    there is a \cnca $f:\cD\to \cU$ and a subsequence
    $(f_{m_j})$ of $(f_m)$ which converges to $f$ 
    uniformly on compact sets. 
\end{lemma}

\def\NN{\mathcal N_\varepsilon}

\begin{proof}
  By hypothesis, there is an
  $\varepsilon>0$ such that
  $\cN_\varepsilon \subset \cD$
  and there is a $C>0$ such that each $X\in\cU$
  satisfies $\|X\|\le C.$
  Each $f_m$ has power series expansion,
\[
   f_m =\sum \hat{f}_m(w) w
\]
  with $\|\hat{f}_m(w)\|\le \frac{C}{\varepsilon^n}$, where
  $n$ is the length of the word $w,$
  by Proposition \ref{prop:rep-nc-map}. Moreover,
  by a diagonal argument, 
  there is a subsequence $f_{m_j}$ of $f_m$
  so that $\hat{f}_{m_j}(w)$ converges to some
  $\hat{f}(w)$ for each word $w$.  Evidently,
  $\|\hat{f}(w)\|\le \frac{C}{\varepsilon^n}$ and thus,
\[
  f=\sum \hat{f}(w) w
\]
  defines a free analytic map on the noncommutative $\frac{\varepsilon}{g}$-neighborhood
  of $0$.  (See \ref{prop:pow-nc-map}.)

  We claim that $f$ determines a free analytic map
  on all of $\cD$ and moreover $(f_{m_j})$ converges
  to this $f$ uniformly on compact sets; i.e., for
  each $n$ and compact set $K\subset \cD(n)$,  the
  sequence $(f_{m_j}[n])$  converges uniformly to $f[n]$ on $K$.

  Conserving notation, let $f_j=f_{m_j}$. 
  Fix $n.$  The sequence $f_j[n]: \cD(n) \to \cD(n)$ is uniformly 
  bounded and hence each subsequence $(g_k)$ 
  of $(f_j[n])$ has a further subsequence 
  $(h_\ell)$ which converges uniformly on compact subsets
  to some analytic function $h:\cD(n)\to \cU(n)$.
  On the other hand, $(h_\ell)$ converges to $\fn$
  on the $\frac{\varepsilon}{g}$-neighborhood of $0$
  in $\cD(n)$ and thus $h=\fn$ on this neighborhood.
  It follows that $\fn$ extends to be analytic on
  all of $\cD(n).$  It
  follows that $(f_j[n])$ itself converges uniformly on compact
  subsets of $\cD(n)$.  In particular, $\fn$ 
  is analytic.

  To see that $f$ is a \emph{free} analytic function (and not just
   that each $f(n)$ is analytic), suppose 
  $X\Gamma = \Gamma Y$.  Then $f_j(X)\Gamma = \Gamma f_j(Y)$ 
  for each $j$ and hence the same is true in the limit. 
\end{proof}

 \begin{lemma}
  \label{lem:nc-aut}
    Suppose $\cD$ 
    is a bounded noncommutative domain which contains a noncommutative neighborhood of $0$. 
    Suppose $(h_n)$ is a sequence from $\Aut_{\rm nc}(\cD)$. If
    $\tilde{h_n}$ converges to $g\in \Aut(\cD(1))$
    uniformly on compact sets, then there is
    $h\in \Aut_{\rm nc}(\cD)$ such that $\tilde{h}=g$ and a subsequence
    $(h_{n_j})$ of $(h_n)$ which converges uniformly on compact sets to $h$. 
 \end{lemma}

\def\hf{\hat{f}}

\begin{proof}
  By the previous lemma, there is a subsequence $(h_{n_j})$ of
  $(h_n)$ which converges uniformly on compact subsets of  $\cD$ to a
  free map $h$.  With 
  $H_j=h_{n_j}^{-1}$, another application of the lemma
   produces a further subsequence, $(H_{j_k})$ which converges
   uniformly on compact subsets of
   $\cD$ to some free map $H$.  Hence, without loss of
   generality, it may be assumed that both $(h_j)$
   and $(h_j^{-1})$ converge (in each dimension) uniformly on compact sets to 
   $h$ and $H$ respectively. 

   From Lemma \ref{lem:scalar-cckw-app}, $\tilde{H}$ is the
   inverse of $\tilde{h}=g$. Thus, letting $f$ denote
   the analytic free  mapping $f=h\circ H$, it follows that
  $\tilde{f}$ is the identity and so by 
   Corollary \ref{cor:cckw1}, $f$ is itself the identity.
   Similarly, $H\circ h$ is the identity.  Thus,
   $h$ is a free biholomorphism 
   and thus an element of $\Aut_{\rm nc}(\cD)$. 
\end{proof}

\begin{proposition}
 \label{prop:tclosed}
  If $\cD$ is a bounded noncommutative domain containing
  an $\varepsilon$-neighborhood of $0$, then the set
  $\{\tilde{h} : h\in \Aut_{\rm nc}(\cD)\}$ 
  is a closed subgroup of $\Aut(\cD(1))$.
\end{proposition}

\begin{proof}
  We must show if $h_n\in\Aut_{\rm nc}(\cD)$ and $\tilde{h_n}$ 
  converges to some $g\in \Aut(\cD(1))$, then there
  is an $h\in\Aut_{\rm nc}(\cD)$ such that $\tilde{h}=g$.
  Thus the proposition is an immediate consequence
  of the previous result, Lemma \ref{lem:nc-aut}.
\end{proof}

\begin{proof}[Proof of Theorem {\rm\ref{thm:nc-circular}}]
  In the BKU Theorem \ref{thm:BKU-really}, first choose $S=\cD(1)$ and let
\[
 G=\{\tilde{f}: f\in\Aut_{\rm nc}(\cD)\}.
\]
  Note that $G$ is a subgroup of $\Aut(S)$ which 
 contains all rotations. Moreover, by Proposition
 \ref{prop:tclosed}, $G$ is closed. 
  Thus Theorem \ref{thm:BKU-really} applies to $G$. 
  Combining the two
  conclusions of the theorem,
   it follows that $G(0)$ is a closed complex submanifold
  of $D$.

  Likewise, let $T=\cU(1)$ and let
\[
 H=\{\tilde{h}: h\in\Aut_{\rm nc}(\cU)\}
\] 
  and note that $H$ is a closed subgroup of $\Aut(T)$
  containing all rotations.   Consequently,
  Theorem \ref{thm:BKU-really} also applies to $H$. 

  Let $\psi:\cD\to \cU$ denote a given  free biholomorphism.
  In particular, $\tilde{\psi}:S\to T$
  is biholomorphic.  Observe,
  $H=\{\tilde{\psi}\circ g \circ \tilde{\psi}^{-1}: g\in G\}$.

    The set
  $\tilde{\psi}(G(0))$ is a closed complex submanifold of $S$, since
  $\tilde{\psi}$ is biholomorphic. On the other hand,
  $\tilde{\psi}(G(0)) = H(\tilde{\psi}(0)).$  Thus,
  by (ii) of Theorem \ref{thm:BKU-really}
  applied to $H$ and $T$, it follows that 
  $\tilde{\psi}(0) \in H(0)$.  Thus, there is
  an $h\in\Aut_{\rm nc}(\cU)$ such that
  $\tilde{h}(\tilde{\psi}(0))=0$. Now  $\varphi=h\circ \psi:\cD\to \cU$
  is a  free biholomorphism between bounded circular
  noncommutative  domains  and $\varphi(0)=0$. Thus,
  $\varphi$ is linear by Theorem \ref{thm:circLin}.
\end{proof}

\subsection{A concrete example of a nonlinear biholomorphic self-map
on an nc LMI Domain}
\label{sec:exYes}
It is surprisingly difficulty to find proper self-maps on LMI domains
which are not linear. In this section we present  
the only (up to trivial modifications)
univariate example,
of which
we are aware.   Of course, by Theorem \ref{thm:circLin} the underlying
domain cannot be circular. In two variables, it can happen that
two LMI domains are linearly equivalent and yet there is a
nonlinear biholomorphism between them taking $0$ to $0$.  We conjecture
this cannot happen in the univariate case.  

Let $A=\begin{bmatrix} 1&1\\ 0&0\end{bmatrix}$
and let $\cL$ denote the univariate $2\times 2$ linear pencil,
$$\cL(x):= I + Ax + A^* x^*
=
\begin{bmatrix} 1 + x +x^* & x \\
                        x^* & 1 \end{bmatrix}.
$$
Let 
$\cD_\cL=\{X: \| X-1 \| < \sqrt 2\}.$
For $\theta\in\R$ consider
$$
f_\theta (x):= \frac{ \e x}{1+x- \e x}.
$$
Then $f_\theta:\cD_\cL\to\cD_\cL$ is a proper
\cncacomman
$f_\theta(0)=0$, and $f^\prime_\theta(0)=\exp(i \theta)$.
Conversely, every proper \cnca  $f:\cD_\cL\to\cD_\cL$ fixing the origin
equals one of the $f_\theta$.

For proofs we refer to \cite[\S 5.1]{HKM3}.

\section{Miscellaneous}\label{sec:misc}

In this section we 
briefly overview some of our other,
more algebraic, results dealing with convexity and LMIs.
While many of these results do have analogs in the present 
  setting of complex scalars and
analytic variables, they appear in the
literature with real scalars and
symmetric free noncommutative variables.

Let $\Rx$ denote the the $\R$-algebra freely generated by $g$
noncommuting letters $x=(x_1,\ldots,x_g)$ with the involution ${}^*$
which, on a word $w\in \cFg,$ reverses  the order; i.e., if
\beq
 \label{eq:w}
  w=x_{i_1} x_{i_2}\cdots x_{i_k},
\eeq
then
\[
  w^* = x_{i_k}\cdots x_{i_2} x_{i_1}.
\]
 In the case $w=x_j,$ note that $x_j^*=x_j$ and for this
  reason we sometimes refer to the variables as 
  {\bf symmetric}.

Let $\SRng$ denote the $g$-tuples $X=(X_1,\dots,X_g)$ of $n\times n$
symmetric real matrices.  A word $w$ as in equation \eqref{eq:w} is
evaluated at $X$ in the obvious way,
\[
 w(X)=X_{i_1}X_{i_2}\cdots X_{i_k}.
\]
  The evaluation extends linearly to polynomials $p\in\Rx$. 
  Note that the  involution on $\Rx$ is compatible with evaluation
  and matrix transpose in that $p^*(X)=p(X)^*$. 

   Given $r$, let $M_r \otimes \Rx$ denote the $r\times r$ matrices
   with entries from $\Rx$.    The evaluation on $\Rx$ extends to 
  $M_r \otimes \Rx$ by simply evaluating entrywise; and the
  involution extends too by $(p_{j,\ell})^* = (p_{\ell,j}^*)$. 

  A polynomial $p\in M_r \otimes \Rx$ is {\bf symmetric} if $p^*=p$ and in this
  case, $p(X)^*=p(X)$ for all $X\in\SRng$.   In this setting, the analog of an LMI is the
  following.  Given $d$ and symmetric $d\times d$ matrices, the symmetric
  matrix-valued degree one polynomial,
\[
   L= I - \sum A_j x_j
\]
 is a {\bf  monic linear pencil.}   The inequality $L(X)\succ
 0$ is then an LMI.    Less formally, the polynomial $L$ itself will
 be referred to as an LMI.

\subsection{nc convex semialgebraic is LMI}\label{sec:convexIsLMI}

\def\posn{\mathcal P[n]}
\def\pos{\mathcal P}
Suppose  $p\in M_r \otimes \Rx$ and $p(0)=I_r$.
  For each positive integer $n$, let
\[
  \pos_p(n) = \{X\in \SRng \colon p(X) \succ0\},
\]
  and define $\pos_p$ to be the sequence
  (graded set) $(\pos_p(n))_{n=1}^\infty$.
   In analogy with classical real algebraic geometry
 we call sets of the form $\pos_p$
 {\bf noncommutative basic open semialgebraic sets}.
  (Note that
  it is not necessary to explicitly consider
  intersections of noncommutative basic open semialgebraic sets
  since the intersection $\pos_p\cap \pos_q$
  equals $\pos_{p\oplus q}$.)

 \begin{theorem}[\cite{HM}]\label{thm:convexIsLMI}
   Every convex bounded noncommutative basic open semialgebraic set
   $\pos_p$ has an LMI representation; i.e., there is a monic 
   linear pencil $L$ such that   $\pos_p=\pos_L.$ 
\end{theorem}

Roughly speaking, Theorem \ref{thm:convexIsLMI} states that
nc semialgebraic and convex equals LMI. Again, this result
is much cleaner than the situation in the classical commutative
case, where the gap between convex semialgebraic and LMI is 
large and not understood very well, cf.~\cite{HV}.

\subsection{LMI inclusion}
The topic of our paper \cite{HKM2} is LMI inclusion
and LMI equality.
Given LMIs $L_1$ and $L_2$
in the same number of variables
it is natural to ask:
\ben[\rm (Q$_1$)]
\item
does one dominate the other, that is,
does
$L_1(X)\succeq0$ imply $L_2(X)\succeq0$?
\item
are they mutually dominant, that is,  do
 they have the same solution set?
\een
As we show in \cite{HKM2}, 
the  domination questions (Q$_1$) and (Q$_2$)
have elegant answers, indeed reduce to semidefinite programs (SDP)
which we show how to construct.
A positive answer to (Q$_1$) is equivalent to the existence
of matrices $V_j$ such that
\beq\label{eq:abstr}
L_2(x)=V_1^* L_1(x) V_1 + \cdots + V_\mu^* L_1(x) V_\mu.
\eeq
 As for (Q$_2$) we show that $L_1$ and $L_2$
  are mutually dominant if and only if, up to
 certain  redundancies described in the paper,
 $L_1$ and $L_2$
 are unitarily equivalent.

A  basic observation is that these 
LMI domination
problems are equivalent
to  the complete positivity of certain
linear maps $\tau$ from a subspace of matrices to
a matrix algebra.

\subsection{Convex Positivstellensatz}\label{sec:Posss}

The equation
\eqref{eq:abstr} can be understood as a linear Positivstellensatz,
i.e., it gives an algebraic certificate for $L_2|_{\cD_{L_1}}\succeq0$.
Our paper \cite{HKM4} greatly extends this to nonlinear $L_2$.
To be more precise,
suppose $L$ is a monic linear pencil in $g$ variables and let $\cD_L$ be 
the corresponding nc LMI.
Then a symmetric noncommutative polynomial $p\in\Rx$ is \emph{positive semidefinite}
on $\cD_L$ if and only if it
has a weighted sum of squares representation with optimal degree bounds.
Namely,
\beq\label{eq:posss}
p
= s^* s   + \sum_j^{\rm finite} f_j^* L f_j,
\eeq
where
 $s, f_j$ are vectors of noncommutative polynomials of degree no greater than
 $\frac{\deg(p) }{2}$. (There is also a bound, coming from 
 a theorem of Carath\' eodory on convex sets in finite dimensional
vector spaces and  depending only on
  the degree of $p$, on the number of terms in the sum.) 
This result contrasts sharply with the commutative setting,
where the degrees of $s, f_j$ are vastly greater than $\deg(p)$
 and assuming only $p$  nonnegative yields a clean
Positivstellensatz so seldom that the cases are noteworthy \cite{Sce}.

The main ingredient of the proof is 
a solution to a noncommutative moment problem, i.e.,
an analysis of
 rank preserving extensions of truncated
noncommutative
 Hankel matrices.
For instance,
 any such \emph{positive definite} matrix $M_k$  of ``degree $k$'' has, for
each $m\geq 0$,
  a positive semidefinite
 Hankel extension $M_{k+m}$ of degree $k+m$ and the same rank
 as $M_k.$
For details and proofs see \cite{HKM4}.

\subsection{Further topics}
  The reader who has made it to this point may be interested in some
  of the surveys, and the references therein,
  on various aspects of noncommutative (free)
  real algebraic geometry, and free positivity.

  The article \cite{HP} treats positive noncommutative polynomials
  as a part of the larger tapestry of spectral theory and 
  optimization.  
In \cite{HKM6} this topic is expanded with further Positivstellens\"atze
and computational aspects.
The survey  \cite{dOHMP} provides a serious
  overview of the connection between noncommutative convexity
  and systems engineering.  The note \cite{HMPV} emphasizes the theme,
  as does the body of this article, 
  that convexity in the noncommutative setting appears to be no
  more general than LMI. 
  Finally, a tutorial with numerous exercises emphasizing the role 
  of the middle matrix and border vector representation 
  of the Hessian of a polynomial  in analyzing convexity 
  is  \cite{HKM5}. 
 
\section*{Acknowledgment and dedication}
 The second and third author appreciate the  opportunity provided
 by this volume to thank Bill for many years of his most generous
 friendship.   Our association with Bill and his many  collaborators
 and friends has had a profound, and decidedly positive,
 impact on our lives, both mathematical and personal.  We are looking
 forward to many ??s to come.

\end{document}